\documentclass[reqno, a4paper, 10pt]{amsart}
\usepackage{amssymb}
\usepackage{mathrsfs}
\usepackage{amssymb, url, color, graphicx, amscd, mathrsfs}
\usepackage[colorlinks=true, bookmarks=true, pdfstartview=FitH, pagebackref=true, linktocpage=true, linkcolor = magenta, citecolor = blue]{hyperref}
\usepackage{graphicx}
\usepackage{times}

\newtheorem{theorem}{Theorem}[section]
\newtheorem{lemma}[theorem]{Lemma}
\newtheorem{corollary}[theorem]{Corollary}

\theoremstyle{definition}
\newtheorem{define}[theorem]{Definition}

\theoremstyle{remark}
\newtheorem{remark}[theorem]{Remark}

\numberwithin{equation}{section}

\allowdisplaybreaks

\begin{document}
\title[Gradient estimate]{Gradient estimates for some evolution equations on complete smooth metric measure spaces}

\author[N.T. Dung]{Nguyen Thac Dung}
\address[N.T. Dung]{Department of mathematics, College of Science\\ Vi\^{e}t Nam National University, Ha N\^{o}i, Vi\^{e}t Nam}
\email{\href{mailto: N.T. Dung <dungmath@gmail.com>}{dungmath@gmail.com}}

\author[K. T. Thuy Linh]{Kieu Thi Thuy Linh}
\address[K. T. Thuy Linh]{Department of mathematics, College of Science\\ Vi\^{e}t Nam National University, Ha N\^{o}i, Vi\^{e}t Nam}
\email{\href{mailto: Linh Kieu <thuylinhkt160991@gmail.com>}{thuylinhkt160991@gmail.com}}

\author[N. V. Thu]{Ninh Van Thu}
\address[N. V. Thu]{Department of mathematics, College of Science\\ Vi\^{e}t Nam National University, Ha N\^{o}i, Vi\^{e}t Nam}
\email{\href{mailto: N. V. Thu <thunv@vnu.edu.vn>}{thunv@vnu.edu.vn}}

\begin{abstract}
In this paper, we consider the following general evolution equation 
$$ u_t=\Delta_fu+au\log^\alpha u+bu $$
on smooth metric measure spaces $(M^n, g, e^{-f}dv)$. We give a local gradient estimate of Souplet-Zhang type for positive smooth solution of this equation provided that the Bakry-\'{E}mery curvature bounded from below. When $f$ is constant, we investigate the gereral evolution on compact Riemannian manifolds with no  
nconvex boundary satisfying an "\emph{interior rolling $R$-ball}" condition. We show a gradient estimate of Hamilton type on such manifolds. 
\end{abstract}


\maketitle


\section{Introduction}
Motivated by the understanding Hamilton's Ricci flow (\cite{Hamilton}), L. Ma introduced in \cite{Lima} the following nonlinear evolution equation
\begin{equation}\label{elima} 
\Delta u+au\log u+bu=0 
\end{equation}
on complete noncompact Riemannian manifolds $(M, g)$ with the Ricci curvature bounded from below. Here  $a<0, b$ are real constants. He obtained a gradient estimate for the positive solution of the evolution equation. Moreover, he pointed out that his result is almost optimal by considering Ricci solitons. To see the relationship between the equation \eqref{elima} and gradient Ricci solitons, let us recall some related notations. A Riemannian manifold $(M, g)$ is said to be a gradient Ricci soliton if there is a smooth function $f$ on $M$ and a constant $\lambda\in\mathbb{R}$ such that
$$ Ric+Hess f=\lambda g. $$
Here $Ric$ is the Ricci tensor of the manifold and $Hess$ is the Hessian with respect to the metric $g$. Gradient Ricci solitons play an important role in Hamilton's Ricci flow as it corresponds to the self-similar soliton and arises as limits of dilations of singularies in the Ricci flow. A Ricci soliton is said to be shrinking, steady, or expanding according to $\lambda$ is positive, zero and negative. Using the contracted Bianchi identity and letting $u=e^{f}$, Ma proved that the equation $Ric+Hess f=\lambda g$ can be deformed to 
$$ \Delta u+ 2\lambda u\log u=(A_0- n\lambda)u $$
where $A_0$ is a constant which can also be determined by choosing extremal point of $f$.

It is raised by Ma an interesting question that consider the local gradient estimate for positive solutions to following
evolution equation
\begin{equation}\label{elima1}
u_t = \Delta u+ au\log u +bu
\end{equation}
in $M$, where $a, b$ are real constants.

Due to the nature of the equations \eqref{elima}, \eqref{elima1} and its potential in studying gradient Ricci solitons, later, \eqref{elima} and \eqref{elima1} are investigated in several papers. For example, in \cite{Yang, Chenchen}, Yang and Chen et al. proved gradient estimates of Li-Yau type for the positive solution of \eqref{elima1} on Riemannian manifold $(M^n, g)$. Consequently, when the Ricci curvature is non-negative, Yang showed that the positive smooth solution of \eqref{1} must be bounded from below by $e^{-n/4}$ (rep. from upper by $e^{n/2}$) if $a<0$ (rep. $a>0$).  (See \cite{Chenchen} for similar results). In \cite{HuangHL}, Huang et al. also pointed out local gradient estimates of the Li-Yau type for positive solutions of \eqref{elima1} on Riemannian manifolds with Ricci curvature bounded from below. As applications, they derived several parabolic Harnack inequalities. Later, in \cite{Cao} Cao et al. considered gradient estimates of positive smooth solution of \eqref{elima1} and applied to study the Ricci flow and gradient Ricci solitons. They proved several differential Harnack inequalities and used them to derive bounds on gradient Ricci solitons. Recently, in \cite{Jiang}, Jiang introduced a local Hamilton type gradient estimate for \eqref{elima1} and obtained a Liouville type theorem for bounded smooth solution of \eqref{elima1}. We refer the reader to \cite{Cao, Chenchen, HuangHL, Jiang, Yang} for further references.   
  
It is worth to notice that gradient Ricci solitons are special cases of the so called smooth metric measure spaces. Recall that an $n-$dimensional smooth metric measure space $(M^n, g, e^{-f} dv)$ is a complete Riemannian manifold $(M^n, g)$ endowed with a weighted measure $e^{-f} dv$ for some $f \in C^{-\infty} (M)$, where $dv$ is the volume element of the metric $g$. On $(M^n,g, e^{-f} dv )$, we have the Bakry-\'{E}mery Ricci tensor 
$$ Ric_f:= Ric + Hess f$$
Many geometric and topological properties for manifolds with Ricci tensor bounded below can be possibly extended to smooth metric measure spaces with the Bakry-\'{E}mery Ricci tensor bounded below. On smooth metric measure spaces, the $f-$Laplacian $\Delta f$ is defined by
$$\Delta_f:= \Delta - \nabla f. \nabla, $$
which is self-adjoint with respect to the weighted measure. Observe that gradient Ricci solitons are smooth metric measure spaces in which $Ric_f=\lambda g$. Therefore, it is also natural to consider modified versions of \eqref{elima}, \eqref{elima1} on such these spaces. Indeed, in \cite{Wu2010}, Wu studied below nonlinear equation
\begin{equation}\label{equa}
u_t=\Delta_fu+au\log u+bu
 \end{equation}
on smooth metric measure spaces $(M^n, g, e^{-f}dv)$ with $m$-dimensional Bakry-\'{E}mery curvature bounded from below. He obtained a local gradient estimate for positive smooth solutions of \eqref{equa}. Then, Wu improved his results by assuming the Bakry-\'{E}mery curvature bounded from below (see \cite{Wu, Wup}). Later, in \cite{Huangma2010, Huangma2015, Dungkhanh}, Huang and Ma, Khanh and the first author studied gradient estimates of Hamilton type and Souplet-Zhang type for the equation \eqref{equa}. As an application, they derived some Liouville type theorems and Harnack inequalities for bounded smooth solution of \eqref{equa}. 

Recently, motivated by studying of \eqref{elima1}, in \cite{Zhuli}, Zhu and Li investigated the following nonlinear evolution equation on complete Riemannian manifolds (with or without boundary) 
$$ \left(\Delta-\frac{\partial }{\partial t}\right)u +au\log^{\alpha}u+bu=0, $$
where $a, \alpha$ are constants and $b$ is ${\mathcal C}^2$ function on $M\times(0, \infty)$. They proved a gradient estimate on complete compact Riemannian manifolds with convex boundary which can be considered a generalization of the famous Li-Yau's gradient estimates (see \cite{LY}). Moreover, they were able to derive gradient estimates of Li-Yau type for positive smooth solution of \eqref{equa} on complete non-compact Riemannian manifolds.
 
In this paper, motivated by studying of \eqref{elima1} and gradient estimates of \eqref{equa}, we consider on smooth metric measure spaces the following general evolution equation
\begin{equation}\label{1}
\Bigl(\Delta_f - \dfrac{\partial}{\partial t} \Bigr) u(x,t) +a u(x,t)(log (u(x,t)))^{\alpha} +bu(x,t)=0,
\end{equation}
where $a, b, \alpha$ are constants. We sometimes write $u(x,t)$ as $u$, also write $\dfrac{\partial u}{\partial t}$ as $u_t$. Our aim is to find gradient estimates of Souplet-Zhang type on complete Riemannian manifolds with or without boundary. We also look for applications of these estimates. Our first main theorem can be read as follows.
\begin{theorem}\label{theomain}
Let $(M, g, e^{-f}dv)$ be an $n$- dimensional complete smooth metric measure space. Suppose that for a fixed point $x_0 \in M$ and $ R \geq 2 $, the $Ric_f\geq-(n-1)K$ with some constant $K \geq 0$ on the $B(x_0, R)$.  Let $a, b$ are constants such that the evolution equation \eqref{1} has a smooth solution  $u$ satisfies $ 1< u \leq D$ with some constant $D>1$ in $Q_{R,T} \equiv B(x_0, R) \times [t_0-T, t_0]$, where $t_0 \in \mathbb{R}$ and $T>0$. For $(x,t) \in Q_{R/2, T}$ with $t \neq t_0 - T$ then
 \begin{enumerate}
\item [(i)]if $\alpha \geq 1$, there exists a constant $c(n)$ such that 
\begin{equation} \label{2}
\dfrac{ \left\vert \nabla u \right\vert}{u} \leq c(n) \left(1+\log \dfrac{D}{u} \right)  \left( \dfrac{1}{R}+ \sqrt{\dfrac{\mu}{R}} + \dfrac{1}{\sqrt{t-t_0+T}}+ \sqrt{K}+\sqrt{P_1} \right),
\end{equation}
where $ P_1=\max \{0, a\} [ \alpha ( \log D)^{\alpha-1} + ( \log D)^{\alpha} ]   + \max \{0, b\}.$
\item  [(ii)]if $\alpha < 1$, there exists a constant $c(n)$ such that 
\begin{equation}
\dfrac{ \left\vert \nabla u \right\vert}{u} \leq c(n)\left(1+\log \dfrac{D}{u}\right) \left( \dfrac{1}{R}+\sqrt{\dfrac{\mu}{R}} + \dfrac{1}{\sqrt{t-t_0+T}}+ \sqrt{K}+\sqrt{P_2}\right),
\end{equation} \label{3}
where $P_2=\max\{0, a \alpha\} \left\vert(\log u)^{\alpha-1}\right\vert_{\infty} + \max\{0, a\} \left\vert(\log u)^{\alpha}\right\vert_{\infty} + \max\{0, b\}$ and noting that $ \left\vert g \right\vert_{\infty}:=\sup_{Q_{R, T}}  \left\vert g \right\vert$.
\end{enumerate}
Here, $\mu: = \max_{\{x|d(x,x_0)=1\}} \Delta_f r(x)$, in that $r(x)$ is the distance function to $x$ from base point $x_0$.
\end{theorem}
As an application, we obtain the below Liouville property.
\begin{theorem}\label{liouville}
Suppose that $(M, g)$ is a complete noncompact smooth metric measure space with nonnegative Bakry-\'{E}mery  curvature. If $a<0$ and $\alpha>0$ then every positive smooth solution $u$ of the equation 
$$ \Delta_f u+au\log^\alpha u=0 $$
is constant provided that $u$ is bounded. Consequently, $u\equiv1$.
\end{theorem}

The second result of this paper is a gradient estimate of Hamilton type for the following evolution equation 
$$ u_t=\Delta u+au\log u+bu $$
on compact Riemannian manifold with nonconvex boundary. Here the solution $u$ is considered to be satisfied the Neumann boundary condition $\dfrac{\partial u}{\partial \nu}=0$, where $\dfrac{\partial}{\partial \nu}$ is the outward pointing unit normal vector to $\partial M$. To state the result, we recall a definition. 
\begin{define}
Let $\partial M$ be the boundary of a compact Riemannian manifold $M$. Then $\partial M$ satisfies the interior rolling $R$-ball condition if for each point $p\in \partial M$ there is a geodesic ball $B_q(R/2)$, centered at some $q\in M$ with radius $R/2$, such that $\{p\}=B_q(R/2)\cap\partial M$ and $B_q(R/2)\subset M$.
\end{define}
Historically, the ``\emph{interior rolling $R$-ball}" condition is traced back to the paper \cite{Chen} by R. Chen. In that paper, Chen gave a lower bound of first Neumann eigenvalue, in his proof, the ``\emph{interior rolling $R$-ball}" condition is used to construct a good cut-off function near the boundary.  It seems that this condition is narutal in studying compact Riemannian manifolds with nonconvex boundary. We refer the reader to an example in \cite{Chen} for the neccesity of this condition in estimating the first Neumann eigenvalue. Later, this condition is also used in \cite{ChenSung, Wang} to derive estimates of the first Stekloff eigenvalue and global heat kernel. Now, we introduce our result.
 \begin{theorem}\label{secondresult}
Let $(M^n, g)$ be a compact Riemannian manifold with boudary $ \partial M $ that satisfies the ``\emph{interior rolling $R$-ball}" condition. Let $K$ and $H$ be nonnegative constants such that the Ricci curvative $Ric_M$ of $M$ is bounded below by $-(n-1)K$ and the second fundamental form $II$ of $ \partial M$ is bounded below $-H$. Suppose that $u\leq D$, for some positive constant $D$, is a positive smooth solution of the equation
\begin{equation}
 \begin{cases}
u_t = \Delta u +au \log u+bu; \\
 \dfrac{\partial u}{\partial \nu}\Bigl\vert_{\partial M}=0.
\end{cases} \label{14}
\end{equation} 
on $M \times (0, \infty)$. By choosing $R$ "small", we have following estimates.
\begin{enumerate}
\item [(i)] If $a > 0$ then
\begin{equation} \label{eq15}
\dfrac{|\nabla u|}{\sqrt{u}} \leq 3 (1+H) \left[ \dfrac{C_1}{R}+ \sqrt{\dfrac{C_2}{R}} + \left(\sqrt{\dfrac{D}{2t}} + C_3 \right)(1+H) \right]. 
\end{equation} 
\item [(ii)] If $a<0$, further assuming that $ \delta \leq u(x,t) \leq D $ for some constant $\delta >0$ then
\begin{equation} \label{eq16}
\dfrac{|\nabla u|}{\sqrt{u}} \leq 3 (1+H) \left[ \dfrac{C_1}{R}+ \sqrt{\dfrac{C_2}{R}} + \left(\sqrt{\dfrac{D}{2t}} + C_4\right)(1+H) \right]. 
\end{equation}
\end{enumerate}
Here, 
\begin{align*}
C_1&= \sqrt[4]{4374 D^2(1+H)^4H^4+ \frac{2}{3} D^2 (17H^2+H)^2};\\
C_2&=\sqrt{\frac{2}{3} \left[ 2D(n-1)H(H+1)(3H+1) \right]^2};\\
C_3&= \sqrt[4]{\frac{2}{3}D^2P^2}, \ P= \max\{2(n-1)K +a(2+\log D)+b, 0\};\\
C_4&= \sqrt[4]{\frac{2}{3} D^2 S^2}, \ S= \max\{2(n-1)K + a(2+ \log \delta)+b, 0 \}.
\end{align*}
\end{theorem}
We would like to mention that the proofs of our theorems mainly are followed by methods in \cite{Bri, Chenchen, SZ, Wang, Wu}. In fact, these methods are standard and well-known. They are used in many works (see also \cite{Bri, Chenchen, ChenSung, Dungkhanh, Huangma2010, Huangma2015, Ru, SZ, Wang, Wu, Wup} and the references therein). More precisely, we first estimate the lower bound of the evolution operator acting on suitable function in term of the evolution solution. Then, we construct a good cut-off function and apply maximal principle to prove the desired results. It is also worth to notice that on compact manifolds with or without boundary, gradient estimates of Li-Yau type are well-studied in literature but up to our knowlegment, gradient estimates of Hamilton type have been not investigated. Therefore, in some sense, our gradient estimates in Theorem \ref{secondresult} are new.

The paper is organized as follows. In the section 2, we give a proof of theorem \ref{theomain} and its corollaries. In the section 3, we prove gradient estimates on compact Riemannian manifolds with nonconvex boundary.
\section{Gradient estimates on noncompact smooth metric measure spaces and its applications}
\setcounter{equation}{0}
In this section, we will give a proof of Theorem \ref{theomain}. First of all, we need to have two technique lemmas as follows. To begin with, let us introduce some notations.

Suppose that $u(x,t)$ is a solution of  $(\ref{1})$ and $1 < u \leq D$ with some constant $D>1$ in $Q_{R,T} \equiv B(x_0, R) \times [t_0-T, t_0]$, where $t_0 \in \mathbb{R}$ and $T>0 $. Let 
$$ h(x,t)= \log \dfrac{u(x,t)}{D}. $$
Then \eqref{1} now can be written as 
\begin{equation}\label{4}
(\Delta_f - \dfrac{\partial}{\partial_t})h +  \left\vert\nabla h \right\vert^2 + a (h+\log D)^{\alpha} + b =0.
\end{equation} 
where $a, b$ are real constants. Clearly, $h \leq 0$.

Now, we introduce the first computational lemma.
\begin{lemma}\label{lemma1}
Let $\omega= \left\vert\nabla\log(1-h) \right\vert^2, (x,t) \in Q_{R,T}$.
 \begin{itemize}
\item[(i)] If $\alpha \geq 1$,  then $ \omega$ satisfies
\begin{align}
\left(\Delta_f - \dfrac{\partial}{\partial_t} \right) \omega \geq &\dfrac{2h}{1-h} \left\langle \nabla h, \nabla \omega \right\rangle +2(1-h) \omega^2 -2 \omega \Bigl[(n-1)K+P_1 \Bigr].\notag
\end{align}
where $P_1:=\max \{ 0, a\} \left(\alpha ( \log D)^{\alpha-1} + ( \log D)^{\alpha} \right)   + \max \{0, b\}$
\item[(ii)] If $\alpha < 1$,  then $ \omega$ satisfies
\begin{align}
\left(\Delta_f - \dfrac{\partial}{\partial_t} \right) \omega  \geq & \dfrac{2h}{1-h}  \left\langle\nabla h, \nabla \omega \right\rangle +2(1-h) \omega^2 - 2\omega \Bigl[ (n-1)K + P_2\Bigr].\notag
\end{align}
where $P_2:=\max\{0, a \alpha\} \left\vert(\log u)^{\alpha-1}\right\vert_{\infty} + \max\{0, a\} \left\vert(\log u)^{\alpha}\right\vert_{\infty} + \max\{0, b\}$.
\end{itemize}
\end{lemma}

\begin{proof}
By Bochner - Weitzenbock formula and the assumption $Ric_f\geq -(n-1)K$, we have
\begin{align}
\Delta_f\omega
&=|\nabla^2\log(1-h)|^2+2\left\langle\nabla\Delta_f\log(1-h),\nabla\log(1-h)\right\rangle\notag\\
&\quad\quad\quad+2Ric_f(\nabla(\log(1-h)), \nabla(\log(1-h)))\notag\\
&\geq -2K(n-1)\omega+2\left\langle\nabla\Delta_f\log(1-h),\nabla\log(1-h)\right\rangle.\label{5}
\end{align} 
We calculate directly that
\begin{align}
\Delta_f\log(1-h)&=\dfrac{-\Delta_fh}{1-h}-\omega=\dfrac{|\nabla h|^2+a(h+\log D)^{\alpha}+b-h_t}{1-h}-\omega \notag\\
&=(1-h)\omega+\dfrac{a(h+\log D)^{\alpha}+b}{1-h}+(\log(1-h))_t-\omega\notag\\
&=\dfrac{a(h+\log D)^{\alpha}+b}{1-h}+(\log(1-h))_t-h\omega.\notag
\end{align}
This equality yields
\begin{align}
2&\left\langle\nabla\Delta_f\log(1-h),\nabla\log(1-h)\right\rangle\notag\\
&=2\left\langle\nabla \Bigl(\dfrac{a(h+\log D)^{\alpha} +b }{1-h}+(\log(1-h))_t - h \omega\Bigl),\nabla\log(1-h)\right\rangle\notag\\
&=-2 a \alpha (h+\log D)^{\alpha-1}\omega-2\dfrac{a(h+\log D)^{\alpha}+b}{1-h}\omega \notag\\
&\quad +\omega_t+2(1-h)\omega^2+\dfrac{2h}{1-h}\nabla \omega \nabla h.\notag
\end{align}
Hence, the inequality \eqref{5} implies
\begin{align}
\Delta_f\omega-\omega_t \geq &-2 \omega \left(a \alpha(h+\log D)^{\alpha -1}+\dfrac{a}{1-h} (h+\log D)^{\alpha} +\dfrac{b}{1-h}\right)\notag\\
&+2(1-h)\omega^2+\dfrac{2h}{1-h}\nabla \omega\nabla h-2(n-1)K\omega. \label{6}
\end{align}
Denote $ S:= a \alpha(h+\log D)^{\alpha -1}+\dfrac{a}{1-h} (h+\log D)^{\alpha} +\dfrac{b}{1-h}. $\\
Since $1-h \geq 1$, we have
\begin{itemize}
\item [(i)] 
if $\alpha \geq 1,$\\
$S \leq \max\{0, a \}\left[\alpha(\log D)^{\alpha -1}+ (\log D)^{\alpha} \right] +\max\{0, b\}.$ \notag
\item [(ii)]
if $ \alpha <1,$\\
$S \leq \max\{0, a \alpha \}\left|\alpha(h+\log D)^{\alpha -1}\right|_{\infty}+ \max\{ 0, a\} \left|(h + \log D)^{\alpha}\right|_{\infty} +\max\{0, b\}.$ \notag
\end{itemize}

Combining $(\ref{4})$ and above two cases, the proof Lemma \eqref{lemma1} is followed.
\end{proof}

Next, we intruduce a smooth cut-off function originated by Li-Yau (\cite{LY}).
\begin{lemma}[\cite{SZ, Wu}] 
\label{lemma2}
Fix $t_0\in\mathbb{R}$ and $T>0$. For any give $\tau\in (t_0-T,t_0]$, there exists a smooth function $\bar\psi:[0,+\infty)\times[t_0-T,t_0]\to\mathbb{R}$ satisfying following propositions
\begin{itemize}
\item [(i)] $0\leq \bar{\psi}(r,t)\leq 1
$ in $[0,R]\times[t_0-T,t_0]$, and it is supported in a subset of $[0,R]\times[t_0-T,t_0]$.
\item [(ii)]
$\bar{\psi}(r,t)=1\text{ and }\bar{\psi}_r(r,t)=0
$ in $[0,R/2]\times[\tau,t_0]$ and $[0,R/2]\times[t_0-T,t_0]$, respectively.
\item [(iii)]
$|\bar{\psi}_t|\leq \dfrac{C\bar\psi^{1/2}}{\tau-t_0+T}
$ in $[0,+\infty)\times[t_0-T,t_0]$ for some $C>0$, and $\bar{\psi}(r,t_0-T)=0$ for all $r\in [0,+\infty)$.
\item [(iv)]
$-\dfrac{C_\epsilon\bar{\psi}^{\epsilon}}{R}\leq\bar{\psi}_r\leq 0 \text{ and }|\bar{\psi}_{rr}|\leq \dfrac{C_\epsilon\bar{\psi}^\epsilon}{R^2}
$ in $[0,+\infty)\times[t_0-T,t_0]$ for every $\epsilon \in (0,1]$ with some constant $C_\epsilon$ depending on $\epsilon$.
\end{itemize}
\end{lemma}

Now, we give a proof of Theorem $\ref{theomain}$ .
\begin{proof}[Proof of Theorem $\ref{theomain}$]
We only prove the case (i) $ \alpha \geq 1$. The case (ii) $\alpha <1$ is proved similarly.
Pick any fixed number $\tau\in (t_0-T,t_0]$, choose a cut-off function $\bar\psi(r,t)$ satisfying propositions of Lemma $\ref{lemma2}$. We shall show that the inequality $(\ref{2})$ holds at every point $(x,\tau)$ in $Q_{R/2,T}$. 

To do this, we introduce a cut-off function $\psi: M\times[t_0-T,t_0]\to \mathbb{R}$ such that 
$$
\psi(x, t)=\bar\psi(d(x,x_0),t),
$$
where $x_0 \in M $ is a fixed point. Consider the function $\psi \omega$ in 
$$Q_{R,T}=\{(x,t)\in M \times [t_0-T, \tau]: d(x,x_0) \leq R\}.$$
Since $\psi\omega$ is continuous, it is obtained the maximal value in $Q_{R, T}$. Let $(x_1, t_1)$ be a maximum for $\psi \omega$ in the set $Q_{R, T}$. We may suppose that $(\psi \omega)(x_1, t_1)>0$; otherwise, if $(\psi\omega)(x_1, t_1) \leq 0$ then $(\psi w)(x,\tau) \leq 0 $ for all $x \in B_{x_0}(R)$. However, by the definition of $\psi$, we have $ \psi(x,\tau)=1$ for all $x \in M$ satifying $d(x,x_0) \leq R/2$ This implies $w(x, \tau ) \leq 0 $ when $x \in B_{x_0}(R)$.  Since $\tau$ is arbitrary, the conclusion then follows.

Since $(\psi \omega)(x_1, t_1) >0$, we infer $t_1\neq t_0-T$. Due to the standard argument of Calabi \cite{Calabi}, we may also assume that $(\psi \omega)$ is smooth at $(x_1, t_1)$. Therefore, at $(x_1, t_1)$, we have  
\begin{equation}\label{max}
\nabla(\psi \omega)=0, \ \Delta_f(\psi \omega)\leq 0, \ (\psi \omega)_t\geq 0.
\end{equation}
By Lemma $\ref{lemma1}$, it follows that
\begin{align}
0  \geq & \left(\Delta_f-\dfrac{\partial}{\partial t}\right)(\psi w)-\left\langle\dfrac{2h}{1-h}\nabla h+2\dfrac{\nabla\psi w}{\psi},\nabla(\psi \omega)\right\rangle\notag\\
=&-2[(n-1)K+P_1](\psi w)+2\psi(1-h)\omega^2-\dfrac{2h}{1-h}\left\langle \nabla\psi,\nabla h\right\rangle w\notag\\&+\omega\Delta_f\psi-\omega\psi_t-2\dfrac{|\nabla \psi|^2}{\psi}\omega.\notag
\end{align}
At $(x_1,t_1)$, the inequality can be simplified as
\begin{align}
2\psi(1-h)\omega^2\leq&\dfrac{2h}{1-h}\left\langle\nabla \psi,\nabla h\right\rangle \omega-\omega\Delta_f\psi+2\dfrac{|\nabla\psi|^2}{\psi}\omega\notag\\
&+\omega\psi_t+ 2[(n-1)K+P_1]\psi \omega.\label{7}
\end{align}
Here we used \eqref{max} to obtain \eqref{7}.

\noindent
{\bf Case 1.} If $x_1\in B(x_0,R/2)$ then by our assumption, $\psi $ is constant in space direction in $B(x_0, R/2)$, where $R \geq 2$. So at $(x_1, t_1)$, the inequality $(\ref{7})$ yields
$$
2\psi \omega^2\leq 2\psi(1-h)\omega^2\leq 2[(n-1)K+P_1]\psi \omega+\omega\psi_t.
$$ 
Here we used $h\leq0$. Note that $\psi(x, \tau)=1$ when $ d(x, x_0) < R/2$,, hence using Lemma \ref{lemma2}, for any $x \in B_{x_0}(R)$, the above estimate indeed implies that
\begin{align*} 
\omega(x,\tau)=& \psi(x, \tau)\omega(x, \tau)\\ 
\leq &(\psi\omega)(x_1, t_1)\leq(\psi^{1/2} \omega) (x_1, t_1)\\ 
\leq & \dfrac{\psi_t}{2\psi^{1/2}}+ [(n-1)K+P_1]\psi^{1/2} \\ 
\leq & \dfrac{C}{\tau-t_0+T}+ [(n-1)K+P_1].
\end{align*}
for all $x \in B(x_0, R/2)$. Since $\tau $ can be arbitrarily choosen, we complete the proof.

{\bf Case 2.} Now, we assume $x_1 \notin B(x_0,R/2)$ where $R\geq 2$. Since $Ric_f\geq -(n-1)K$ and $r(x_1,x_0)\geq 1$ in $B(x_0,R)$, the $f$-Laplacian comparison theorem in \cite{Bri} gives
\begin{equation}
\Delta_fr(x_1)\leq \mu +(n-1)K(R-1),\notag
\end{equation}
where $\mu:=\max_{x\in B(x_0,1)}\Delta_fr(x).$
In the following computations, $c$ denotes a constant depending only on $n$ whose value may change from line to line.

Using the above $f$-Laplacian comparison theorem and Lemma \ref{lemma2}, we have
\begin{align}
-\omega\Delta_f\psi=&-\big[\psi_r\Delta_fr+\psi_{rr}|\nabla r|^2\big]\omega\notag\\
\leq&-\big[\psi_r\big(\mu+(n-1)K(R-1)\big)+\psi_{rr}\big]\omega\notag\\
\leq&\omega\psi^{1/2}\dfrac{|\psi_{rr}|}{\psi^{1/2}}+|\mu|\psi^{1/2}\omega\dfrac{|\psi_r|}{\psi^{1/2}}+\dfrac{(n-1)K(R-1)|\psi_r|}{\psi^{1/2}}\psi^{1/2}\omega\notag\\
\leq&\dfrac{1}{8}\psi \omega^2+c\left[\left(\dfrac{|\psi_{rr}|}{\psi^{1/2}}\right)^2+\left(\dfrac{|\mu||\psi_r|}{\psi^{1/2}}\right)^2+\left(\dfrac{K(R-1)|\psi_r|}{\psi^{1/2}}\right)^2\right]\notag\\
\leq&\dfrac{1}{8}\psi \omega^2+c\left(\dfrac{1}{R^4}+\dfrac{|\mu|^2}{R^2}+K^2\right).\label{8}
\end{align}
Now, we use the Young's inequality to estimate the terms of the right hand side of \eqref{7}. First, we have the estimates of first term of the right hand side of \eqref{7}
\begin{align}
\dfrac{2h}{1-h}\left\langle\nabla\psi,\nabla h\right\rangle \omega\leq &2|h||\nabla\psi|\omega^{3/2}\notag\\
=&2\big[\psi(1-h)\omega^2\big]^{3/4}\dfrac{|h||\nabla\psi|}{\big[\psi(1-h)\big]^{3/4}}\notag\\
\leq&\psi(1-h)\omega^2+c\dfrac{(h|\nabla\psi|)^4}{\big[\psi(1-h)\big]^3}\notag\\
\leq&\psi(1-h)\omega^2+c\dfrac{h^4}{R^4(1-h)^3}.\label{9}
\end{align}
For the third term of the right hand side of \eqref{7}
\begin{align}  
2\dfrac{|\nabla\psi|^2}{\psi}\omega 
& =2\psi^{1/2}\omega\dfrac{|\nabla\psi|^2}{\psi^{3/2}}  \notag\\ 
& \leq \dfrac{1}{8}\psi \omega^2+c\left(\dfrac{|\nabla\psi|^2}{\psi^{3/2}}\right)^2 \notag\\
& \leq \dfrac{1}{8}\psi \omega^2+\dfrac{c}{R^4}. \label{10}
\end{align}
For the forth term of the right hand side of \eqref{7} 
\begin{align}
\omega\psi_t=\psi^{1/2}\omega\dfrac{\psi_t}{\psi^{1/2}} 
& \leq \dfrac{1}{8}\psi \omega^2+c\left(\dfrac{\psi_t}{\psi^{1/2}}\right)^2  \notag\\ 
& \leq\dfrac{1}{8}\psi \omega^2+\dfrac{c}{(\tau-t_0+T)^2}. \label{11}
\end{align}
Finally, we estimate the last term 
\begin{equation}
 2[(n-1)K+P_1]\psi \omega\leq \dfrac{1}{8}\psi \omega^2+ c(K^2+(P_1)^2). \label{12}
\end{equation}
Substituting $(\ref{8})$-$(\ref{12})$ into $(\ref{7})$ we get
\begin{align}
2 \psi(1-h)\omega^2\leq& \psi(1-h)\omega^2++\dfrac{1}{2}\psi \omega^2+ \dfrac{h^4c}{R^4(1-h)^3} +\dfrac{c}{R^4}\notag\\
&+\dfrac{\mu^2c}{R^2}+\dfrac{c}{(\tau-t_0+T)^2} + cK^2 + c(P_1)^2. \label{13}
\end{align}
Recall that $1-h \geq 1$, then $(\ref{13})$ implies
\begin{align*}
\psi \omega^2\leq c\left(\dfrac{h^4}{R^4(1-h)^4}+\dfrac{1}{R^4}+\dfrac{\mu^2}{R^2}+\dfrac{1}{(\tau-t_0+T)^2}+K^2+(P_1)^2 \right).
\end{align*} 
Morever, since $\psi(x,\tau)=1$ with $x$ in $B(x_0,R/2)$ and $ h^4 / (1-h)^4 \leq 1$, then
$$
\omega^2(x,\tau)\leq \psi \omega^2(x_1,t_1)\leq c\left(\dfrac{1}{R^4}+\dfrac{\mu^2}{R^2}+\dfrac{1}{(\tau-t_0+T)^2}+K^2+P_1^2 \right).
$$
Since $\omega(x, \tau)$ is defined with $\tau $ arbitrary in $(t_0-T, t_0]$, so for all $(x,t)$ in $Q_{R/2,T}$ we have
$$ \dfrac{|\nabla h|}{(1-h)} (x, t) \leq c\left(\dfrac{1}{R^2}+ \dfrac{|\mu|}{R}+ \dfrac{1}{\sqrt{t-t_0+T}}+ \sqrt{K}+\sqrt{P_1} \right).
$$
Therefore, we have finished the proof of theorem since $h=\log (u/D)$, and hence the proof is complete.
\end{proof}
\begin{proof}[Proof of Theorem \ref{liouville}]Suppose that $u$ is a positive smooth solution of the equation
$$ \Delta_fu+au\log^\alpha u=0. $$
Moreover, $u\leq D$ for some $D>1$. Note that if $b=0, \alpha>0>a$ then $P_1=P_2=0$. Since $u$ is independent of $t$, we may let $t$ tending to infinity in Theorem \ref{theomain}, then letting $R\to\infty$, we obtain, for $K=0$
$$ |\nabla u|=0. $$
Hence, $u$ is constant. This forces $u\equiv1$.
\end{proof}

\begin{corollary} \label{co2.3}
(Harnack - type inequality) Let $(M, g, e^{-f}dv)$ be an $n -$ dimensional complete smooth metric measure space with $Ric_f \geq -(n-1)K$ for some constant $K \geq 0 $ in $B(x_0, R)$, fixed $x_0$ in $M$ and $R \geq 2$. Suppose that $u$ is a smooth solution to equation $(\ref{1})$ satisfies $1 <u \leq D$ with some constant $D>1$. Assume that $\rho =\rho (x_1, x_2)$ is geodesic distance between $x_1$ and $x_2$ for all $x_1, x_2$ in $M$, we have
\begin{enumerate}
\item[(i)] if $\alpha \geq 1$ then
$$ u(x_2, t) \leq u(x_1, t)^{\beta_1} (eD)^{1-\beta_1},$$
where $ \beta_1=exp \left( - \dfrac{c(n) \rho}{\sqrt{t-t_0+R}}-c(n)\sqrt{K} \rho -c(n) \sqrt{P_1} \rho \right)$, $c(n)$ is a constant depending to $n$.
\item[(ii)] if $\alpha < 1$ then
$$ u(x_2, t) \leq u(x_1, t)^{\beta_2} (eD)^{1-\beta_2},$$
where $ \beta_2=exp \left( - \dfrac{c(n) \rho}{\sqrt{t-t_0+R}}-c(n)\sqrt{K} \rho -c(n) \sqrt{P_2} \rho \right)$, $c(n)$ is a constant depending to $n$.
\end{enumerate}
\end{corollary}

\begin{proof}[Proof of Corollary \ref{co2.3}]
We prove Corollary \ref{co2.3} for the case $\mathrm{(i)}$ \ $\alpha \geq 1$. For the case $\mathrm{(ii)}$ $ \alpha <1$, the proof is similar. By the estimates $(\ref{2})$ and let $R $ tends to infinity, we obtain
$$ \dfrac{|\nabla u|}{u(1+\log \dfrac{D}{u})}  \leq c(n) \left( \dfrac{1}{\sqrt{t-t_0+T} } +\sqrt{K} + \sqrt{P_1} \right).$$
Let $\gamma : [0,1] \to M$ be a minimal geodesic joining $x_1$ and $x_2$ satisfying $\gamma(0)=x_2$ and $\gamma(1)=x_1$. Then 
\begin{align*}
\log\dfrac{1-h(x_1,t)}{1-h(x_2,t)} = & \int\limits_0^1 \dfrac{d \log (1-h(\gamma(s),t))}{ds} ds \\ \leq & \int\limits_0^1 |\dot{\gamma}| \dfrac{|\nabla u|}{u(1+\log \dfrac{D}{u})} ds \\ \leq & c(n) \left( \dfrac{1}{\sqrt{t-t_0+T}} +\sqrt{K} + \sqrt{P_1} \right) \rho.
\end{align*}
Let $\beta_1 = exp \left( - \dfrac{c(n) \rho}{\sqrt{t-t_0+T}}-c(n)\sqrt{K} \rho -c(n) \sqrt{P_1} \rho \right)$, we have 
$$\dfrac{1-h(x_1, t)}{1-h(x_2, t)} \leq \dfrac{1}{\beta_1}.$$
With some easy caculations, we get 
$$ u(x_2, t) \leq u(x_1, t)^{\beta_1} \left( eD \right)^{1-\beta_1}.$$
Hence, Corollary \ref{co2.3} is proved.
\end{proof}
\section{Gradient estimates on compact Riemannian manifolds with non-convex boundary}
Let $M$ be a compact Riemannian manifold with nonconvex boundary. Let $\partial /\partial \nu$ be the outward pointing unit normal vector to $\partial M$ and let $II$ be the second fundamental form of $\partial M$ with respect to $\partial /\partial \nu$. In the seminal paper \cite{LY}, Li--Yau proved that if $M$ is a compact Riemannian manifolds with nonnegative Ricci curvature and the boundary $\partial M$ is convex in the sense that $II\geqslant 0$, then any nonnegative solution $u$ of the evolution equation $\Delta u-\partial_tu=0$ on $M\times(0, +\infty)$ with Neumann boundary condition $\partial_\nu u=0$ satisfies
$$\frac{|\nabla u|^2}{u^2}-\frac{u_t}{u}\leqslant \frac{n}{2t}$$
on $M\times(0, +\infty)$. 
Later, that Li-Yau gradient estimates was generalized to manifolds with non-convex boundary by Chen \cite{Chen} and Wang \cite{Wang}. Here the boundary $\partial M$ is said to be non-convex if there exists a positive constant $H$ such that the second fundamental form $II$ is bounded from below by $-H$. Due to the non-convexity of the boundary, we have to control some technical computations since the estimates necessarily involve the second fundamental form of $\partial M$. As in \cite{Chen, Wang}, we need to construct a good cut-off function near $\partial M$, this turns out that the ``\emph{interior rolling $R$-ball condition}" are neccesary added. Furthermore, we would like to note that the interior rolling $R$-ball condition is a geometric condition on the boundary $\partial M$ to ensure that the first Neumann eigenvalue is bounded away from zero (see \cite{Chen}) and that the second fundamental form is bounded from above (see \cite{ChenSung}).  

Assume that $u$ is a positive solution of the evolution equation 
\begin{equation}\label{ecom} 
u_t=\Delta u+au\log u+bu, 
\end{equation}
where $a, b\in\mathbb{R}$ are fixed constant.
To prove the theorem \ref{secondresult}, as in \cite{Jiang, Wu}, let 
$$h(x,t):= u^{1/3}(x,t).$$
Then the equation \eqref{ecom} becomes
\begin{equation}
 h_t =\Delta h + 2 h^{-1} |\nabla h|^2 +ah \log h + \dfrac{b}{3} h \label{16}
\end{equation}
Using the same strategy as in the proof of Theorem \ref{theomain}, we first derive the following computational lemma. 
\begin{lemma} \label{lemma3}
Let $ \omega:= h |\nabla h|^2 $. For any $(x,t) \in M \times (0, \infty)$, 
\begin{enumerate}
\item [(i)] if $a \geq 0, $ then $\omega$ satifies
$$ (\Delta - \partial_t)\omega \geq \dfrac{9}{2} h^{-3} \omega^2 -3 h^{-1} \left\langle \nabla \omega, \nabla h \right\rangle- [2(n-1)K + a\log D +2a +b]\omega.$$
\item [(ii)] if $a<0$, further assuming that  $0< \delta \leq u(x,t) \leq D$ for some constant $\delta >0$, then $\omega$ satisfies
$$ (\Delta - \partial_t)\omega \geq \dfrac{9}{2} h^{-3} w^2 -3 h^{-1} \left\langle \nabla \omega, \nabla h\right\rangle - [2(n-1)K + a \log \delta +2a +b]\omega.$$
\end{enumerate}
\end{lemma}
 
\begin{proof}
For any smooth function $v$, the Bochner - Weitzenb\"{o}ck formula gives
$$ \Delta |\nabla v|^2 \geq 2|\nabla^2 v|^2+ 2 Ric \left\langle \nabla v, \nabla v \right\rangle+2 \left\langle \nabla \Delta v, \nabla v \right\rangle. $$
Considering the function $v = \dfrac{2}{3}h^{3/2} $ then  $\omega= \left\vert \nabla v \right\vert^2$. The assumption $Ric \geq -(n-1)K $ and the Bochner - Weitzenb\"{o}ck formular imply
\begin{equation*}
\Delta \omega \geq -2(n-1)K \omega + 2 \left\langle \nabla\Delta \left( \dfrac{2}{3} h^{3/2} \right), \nabla \left( \dfrac{2}{3} h^{3/2} \right) \right\rangle.
\end{equation*} 
By $(\ref{16})$ and a direct computation shows that
\begin{align*}
\Delta\left(\dfrac{2}{3} h^{3/2} \right) 
& = h^{1/2} \Delta h + \dfrac{1}{2} h^{-1/2} \left\vert \nabla h \right\vert^2 \\ 
& = h^{1/2} h_t -\dfrac{3}{2} h^{-3/2} \omega -ah^{3/2} \log h - \dfrac{b}{3} h^{3/2}.
\end{align*} 
This equality yields
\begin{align}
2 & \left\langle \nabla\Delta \left( \dfrac{2}{3} h^{3/2} \right), \nabla \left( \dfrac{2}{3} h^{3/2} \right) \right\rangle \\
 & =  2 \left\langle \nabla \left( h^{1/2} h_t -\dfrac{3}{2} h^{-3/2} \omega -ah^{3/2} \log h - \dfrac{b}{3} h^{3/2} \right), h^{1/2}\nabla h \right\rangle\notag \\
 &=  \omega_t +\dfrac{9}{2} h^{-3} \omega^2 -3h^{-1} \left\langle \nabla \omega, \nabla h\right\rangle -a\omega\left(2+ 3 \log h \right) -b\omega.\label{eli}
\end{align}
Thus, we obtain
$$ \Delta \omega -\omega_t \geq \dfrac{9}{2} h^{-3} \omega^2 -3h^{-1} \left\langle \nabla \omega, \nabla h\right\rangle - \omega\left[ 2(n-1)K +a\left(2+ 3 \log h \right)+ b\right].$$
To conclude the proof, we notice that the following two cases hold true. 
\begin{itemize}
\item If $a \geq 0$ and $0<h \leq D^{1/3}$, then $\log h \leq 1/3 \ \log D .$
\item If $a<0$ and $ \delta^{1/3} \leq h \leq D^{1/3}$, then $ 1/3 \log \delta \leq \log h \leq 1/3 \log D.$ 
\end{itemize}
Combining the above two cases and the inequality \eqref{eli}, we obtain the desired results and the Lemma is proved completely.
\end{proof}
\begin{proof} [Proof of Theorem \ref{secondresult}]
We only consider the case $a \geq 0$. The case $a <0$ is proved similarly.

Suppose $ r(x)$ is the distance from $x \in M $ to $\partial M$. As in \cite{Chen}, we define a function on $M$ by $ \phi (x)= \psi \left(\dfrac{r(x)}{R} \right),$ where $ \psi (x)$ is nonnegative ${\mathcal C}^2$-function defined on $[0, \infty]$ such that 
$$ \begin{cases}
\psi(r) \leq H&\text{ if }  0 \leq r < 1/2,\\
 \psi (r)=H&\text{ if }   1 \leq  r < \infty.
\end{cases} $$
with  $\psi (0)=0, \ \ 0 \leq \psi'(r)  \leq 2H, \ \ \psi'(0)=H$ and $\psi''(r) \geq -H$. Put
$$\chi(x) = \left(1+\phi(x) \right)^2 $$
and let 
$$F(x, t)= t \chi \omega=t \chi  h |\nabla h|^2. $$
For any fixed $T < \infty$, $F(x,t)$ is continuous on $\overline{M} \times [0, T]$. We can suppose that $(x_0, t_0) \in \overline{M} \times [0, T] $ be a maximum point for the function $F$. If $F(x_0, t_0)=0$ then the right hand side of \eqref{eq15} and of \eqref{eq16} is zero at $(x, T) \in M \times (0, T]$. Since $T $ is arbitrary, the inequalities in Theorem \ref{secondresult} are trivial. Hence, we can assume that $F(x_0, t_0) >0$. Consequently, $t_0\not=0$.

If $x_0 \in \partial M$ then $\dfrac{\partial F}{ \partial \nu} (x_0, t_0) \geq 0 $. Let $e_1, e_2, \cdots, e_n$ be an orthormal frame at $x_0$ with $e_n = \nu.$ We have
$$ 0 \leq \dfrac{\partial F}{\partial \nu} (x_0, t_0) = t_0 \left( \dfrac{\partial \chi (x_0)}{\partial \nu} h |\nabla h|^2 + \chi(x_0) h_{\nu} |\nabla h|^2 + \chi (x_0) 2 h  \sum\limits_{i=1}^n h_i h_{i \nu} \right).$$
Further, $h_n = h_{\nu} = \dfrac{\frac{\partial u}{\partial \nu} }{3 u^{2/3}} = 0$ on $\partial M.$
Since $t_0>0$, we get 
$$ \dfrac{\partial \chi(x_0)}{\partial \nu} \dfrac{1}{\chi(x_0)} + \dfrac{ 2 \sum\limits_{i=1}^n h_i h_{i \nu} }{|\nabla h|^2} \geq 0.$$
On the other hand, by a direct computation, one shows that
$$ \sum\limits_{i=1}^n h_i h_{i \nu} = \left\langle \nabla h, \left( \nabla h\right)_{\nu} \right\rangle = -II \left\langle \nabla h, \nabla h \right\rangle \leq H |\nabla h|^2$$
and 
$$ \dfrac{\partial \chi(x_0)}{\partial \nu} \dfrac{1}{\chi(x_0)} =-\dfrac{2H}{R}.$$
If we choose $R<1$ then  
$$\dfrac{\partial \chi(x_0)}{\partial \nu} \dfrac{1}{\chi(x_0)} + \dfrac{ 2  \sum\limits_{i=1}^n h_i h_{i \nu} }{|\nabla h|^2} \leq -\dfrac{2H}{R} + 2H<0.$$
This is contradiction.

Now, we can assume that F achieves its at $ (x_0, t_0) \in (\overline{M} \setminus \partial M) \times [0, T]$  and at $(x_0, t_0)$ then $ \nabla F =0, \dfrac{\partial F}{\partial t} \geq 0, \Delta F \leq 0$. We infer at $(x_0, t_0)$ 
\begin{align*}
 0  \geq   \Delta F - F_t  
 & = t_0\omega \Delta \chi  + t_0 \chi \Delta \omega + 2t_0 \left\langle \nabla \chi, \nabla \omega\right\rangle - t_0 \chi \omega_t - \chi \omega \\ 
 & =  t_0 \chi \left( \Delta \omega - \omega_t \right) + t_0\omega \Delta \chi +2t_0 \left\langle \nabla \chi, \nabla \omega \right\rangle- \chi \omega.
\end{align*}
By Lemma $\ref{lemma3}$, we obtain 
\begin{align}
 0  \geq  & \  \Delta F - F_t \notag\\
  \geq  &  t_0 \chi \left( \dfrac{9}{2} h^{-3} \omega^2 -3 h^{-1} \left\langle \nabla h,  \nabla \omega \right\rangle - \omega P \right) + t_0\omega \Delta \chi 
+ \ 2t_0 \left\langle \nabla  \chi, \nabla \omega\right\rangle - \chi \omega.\label{de1}
\end{align}
Since $ 0= \nabla F = t_0\omega \nabla \chi + t_0 \chi\nabla \omega$, we have at $(x_0, t_0)$ 
\begin{equation}\label{de2} \begin{cases}
\left\langle \nabla \chi, \nabla \omega \right\rangle = - \dfrac{ |\nabla \chi|^2}{\chi } \omega \geq -16\dfrac{H^2}{R^2} \omega ,\\
\left\langle \nabla \omega, \nabla h \right\rangle = - \left\langle \dfrac{ \nabla \chi}{\chi }, \nabla h \right\rangle \omega.
\end{cases} 
\end{equation}
We let $ \partial M (R) = \{ x \in M \bigl| r(x) \leq R\}$, by using an index comparision theorem in \cite{Wa} (see also \cite{Chen}), we have 
$$ \Delta r(x) \geq -(n-1)(3H+1)$$
for $x\in \partial M(R).$ Therefore,
\begin{align*} 
\Delta \phi 
&= \dfrac{1}{R} \psi' \Delta r + \dfrac{1}{R^2} \psi'' |\nabla r|^2\\
& \geq -\dfrac{2(n-1)H(3H+1)}{R} - \dfrac{H}{R^2}.
\end{align*}
With the help of this Laplacian comparison, we can estimate 
\begin{align} 
\Delta \chi 
&= 2(1+ \phi) \Delta \phi + 2 |\nabla \phi |^2 \notag\\
&\geq 2(1+H) \left( - \dfrac{2(n-1)H(3H+1)}{R} -\dfrac{H}{R^2} \right).\label{de3}
\end{align}
Plugging \eqref{de2}, \eqref{de3} into \eqref{de1}, we obtain
\begin{align*}
0  
\geq & \dfrac{9}{2} t_0 \omega^2 \chi h^{-3}  + 3 t_0 \omega h^{-1} \left\langle \nabla \chi, \nabla h \right\rangle  - P (1+H)^2 t_0 \omega \\ 
& + t_0 \omega 2(1+H)\left( - \dfrac{2(n-1)H(3H+1)}{R} - \dfrac{H}{R^2} \right) - 32t_0\omega \dfrac{H^2}{R^2} - (1+H)^2 \omega \\ 
\geq & \dfrac{9}{2} t_0 \omega^2 h^{-3} + 3 t_0\omega h^{-1} \left\langle \nabla \chi, \nabla h \right\rangle  - (1+H)^2 \omega  \\ 
& -2 t_0 \omega \left[ \dfrac{P(1+H)^2 }{2}  +(1+H)\left( \dfrac{2(n-1)H(3H+1)}{R} +\dfrac{H}{R^2} \right) +16\dfrac{H^2}{R^2} \right] \\  
\geq & \dfrac{9}{2} t_0 \omega^2 h^{-3} +3t_0 \omega h^{-1} \left\langle \nabla \chi, \nabla h \right\rangle -(1+H)^2 \omega -2t_0 \omega Q,
\end{align*}
where 
$$Q:= \dfrac{P(1+H)^2}{2} +\dfrac{2(n-1)H(H+1)(3H+1)}{R} +\dfrac{17H^2+H}{R^2}.$$
Multipliting two hand side of the above inequality with $ h^3$, we have 
\begin{equation} \label{17}
9t_0\omega^2 \leq -6t_0\omega h^2 \left\langle \nabla \chi, \nabla h \right\rangle +2 (1+H)^2 \omega h^3 + 4 Q t_0 \omega h^3.
\end{equation} 
Apply the Young's inequality, we obtain the following two estimates 
\begin{enumerate}
\item[ (i)] 
\begin{align} \label{18}
-6t_0\omega h^2 \left\langle \nabla \chi ,\nabla h \right\rangle 
& \leq 6t_0 h^{3/2} \left\vert \nabla  \chi \right\vert \omega^{3/2} \notag \\ 
& =\leq \dfrac{8}{3}t_0 \left[\omega^{3/2} \dfrac{9}{4} D^{1/2} \left\vert \nabla  \chi \right\vert \right] \notag \\ 
& \leq 2t_0\omega^2 +4374  \dfrac{t_0D^2 H^4(1+H)^4}{R^4}.
\end{align}
\item[(ii)]
\begin{align} \label{19}
2 (1+H)^2 \omega h^3 
& =  4 \left[ \sqrt{t_0}\omega  \dfrac{h^3(1+H)^2}{2 \sqrt{t_0}} \right] \notag \\  
& \leq 2t_0 \omega^2 + \dfrac{(1+H)^4 h^6}{2t_0} \notag \\ 
& \leq 2t_0\omega^2 +\dfrac{(1+H)^4 D^2 }{2t_0}.
\end{align}
\end{enumerate}
Finally, we use Cauchy's inequality to obtain
\begin{align} \label{20}
4t_0\omega Qh^3  & \leq 2t_0\omega^2+ 2t_0 Q^2 D^2.
\end{align}
Plugging the \eqref{18} - \eqref{20} into \eqref{17}, we infer at $(x_0,t_0)$ 
$$ 9t_0\omega^2 \leq 8t_0\omega^2 + 4374 \dfrac{ t_0D^2 (1+H)^4 H^4}{R^4}+ \dfrac{(1+H)^4 D^2}{2t_0} + 2 t_0 Q^2 D^2 $$
Consequently,
\begin{align*}
\omega^2  \leq 
& 4374 \dfrac{D^2(1+H)^4H^4}{R^4}  +\dfrac{D^2 (1+H)^4}{2t_0^2} + 2Q^2 D^2 \\  
= & \dfrac{1}{R^4} \left[4374 D^2(1+H)^4H^4+ \frac{2}{3} D^2 (17H^2+H)^2 \right]  +\dfrac{D^2 (1+H)^4}{2t_0^2} \\
& +  \dfrac{2}{3}D^2P^2(1+H)^4 + \dfrac{2}{3R^2}  \left[ 2D(n-1)H(H+1)(3H+1) \right]^2\\ 
= & \dfrac{(C_1)^4}{R^4}+ \dfrac{(C_2)^2}{R^2}+\dfrac{D^2 (1+H)^4}{2t_0^2}+  (C_3)^4(1+H)^4.
\end{align*} 
Therefore, for any $x \in \overline{M}$,
\begin{align*}
T\omega(x,T) 
\leq & T(1+\phi(x))^2 w(x, T) \\
\leq & t_0(1+\phi(x_0))^2 w(x_0, t_0) \\
\leq & (1+H)^2 \left[ t_0 \left( \dfrac{(C_1)^2}{R^2}+ \dfrac{C_2}{R} \right) +\dfrac{D(1+H)^2}{\sqrt{2}} + t_0(C_3)^2(1+H)^2 \right]  \\
\leq & (1+H)^2 \left[T \left( \dfrac{(C_1)^2}{R^2}+ \dfrac{C_2}{R} \right) +\dfrac{D(1+H)^2}{\sqrt{2}} +T (C_3)^2(1+H)^2 \right] .
\end{align*}
This implies that 
$$w(x, T) \leq (1+H)^2 \left[ \dfrac{(C_1)^2}{R^2} + \dfrac{C_2}{R} + \dfrac{D(1+H)^2}{\sqrt{2}T} + (C_3)^2(1+H)^2\right].$$
Since $T$ is arbitrary, the proof is complete.
\end{proof}
\begin{remark}As in \cite{Chen, Wang}, in our estimate, the notation $R$ is "small" maens $R$ is chosen to be a positive constant less than $1$. Moreover, $R$ depends on the upper bound of the sectional curvature of the
manifold near the boundary. More precisely, the upper bound of $R$ is determined by
$$ \sqrt{K_R}\tan(R\ \sqrt{K_R})\leq\frac{H}{2}+\frac{1}{2} $$
and
$$ \frac{H}{\sqrt{K_R}}\tan(R\ \sqrt{K_R})\leq\frac{1}{2}, $$
where $K_R$ is the upper bound of the sectional curvature on the set $M_R=\left\{x\in M: r(x)\leq R\right\}$.
\end{remark}
As applications, first we have following gradient estimates.
 \begin{theorem}\label{main3}
Let $(M^n, g)$ be a compact Riemannian manifold with convex boudary $ \partial M $ that satisfies the " interior rolling $R$-ball" condition. Let $K$ be nonnegative constant such that the Ricci curvative $Ric_M$ of $M$ is bounded below by $-K$. Suppose that $u\leq D$, for some positive constant $D$, is a positive smooth solution of the equation
\begin{equation}
 \begin{cases}
u_t = \Delta u +au \log u+bu; \\
 \dfrac{\partial u}{\partial \nu}\Bigl\vert_{\partial M}=0.
\end{cases} \label{14}
\end{equation} 
on $M \times (0, \infty)$. By choosing $R$ "small", we have following estimates.
\begin{enumerate}
\item [(i)] If $a > 0$ then
\begin{equation} \label{e15}
\dfrac{|\nabla u|}{\sqrt{u}} \leq 3\left(\sqrt{\dfrac{D}{2t}} + C_3 \right). 
\end{equation} 
\item [(ii)] If $a<0$, further assuming that $ \delta \leq u(x,t) \leq D $ for some constant $\delta >0$ then
\begin{equation} \label{e16}
\dfrac{|\nabla u|}{\sqrt{u}} \leq 3\left(\sqrt{\dfrac{D}{2t}} + C_4\right). 
\end{equation}
\end{enumerate}
Here 
\begin{align*}
C_3&= \sqrt[4]{\frac{2}{3}D^2P^2}, \ P= \max\{2K +a(2+\log D)+b, 0\};\\
C_4&= \sqrt[4]{\frac{2}{3} D^2 S^2}, \ S= \max\{2K + a(2+ \log \delta)+b, 0 \}.
\end{align*}
\end{theorem}
\begin{proof}Since $\partial M$ is convex, we have $H=0$. Therefore, $C_1=C_2=0$. The proof of Theorem \ref{main3} is followed directly by Theorem \ref{secondresult}.
\end{proof}
\begin{corollary}
Let $(M^n, g)$ be a compact Riemannian manifold with convex boudary $ \partial M $ that satisfies the " interior rolling $R$-ball" condition. Suppose that the Ricci curvature is nonnegative and $u$ is a positive smooth solution of the equation
\begin{equation}
 \begin{cases}
\Delta u +au \log u+bu=0 \\
 \dfrac{\partial u}{\partial \nu}\Bigl\vert_{\partial M}=0.
\end{cases} \label{14}
\end{equation} 
on $M \times (0, \infty)$. By choosing $R$ "small", we have following estimates.
\begin{enumerate}
\item [(i)] If $a > 0\geq b$ then $u\geq e^{-2}$.
\item [(ii)] If $a<0$ and $b\leq0$, further assuming that $ e^{-2} \leq u(x,t) \leq D $ then $u$ is constant. Consequently, $u=e^{-b/a}$.
\end{enumerate}
\end{corollary}
\begin{proof}The proof is followed by Theorem \ref{main3}. Indeed, since $u$ does not depend on $T$, we may let $T\to\infty$ in Theorem \ref{main3} to have following two cases.
\begin{enumerate}
\item If $a > 0\geq b$ and $u\leq e^{-2}$ then by Theorem \ref{main3}, we have that $|\nabla u|=0$. Consequently, $u=e^{-b/a}>1$. This is a contradiction since $u\leq e^{-2}$.
\item If $a<0, b\leq0$ and $ e^{-2} \leq u(x,t) \leq D$ then by Theorem \ref{main3}, we infer $|\nabla u|=0$, namely $u$ is constant. Hence, $u=e^{-b/a}$. 
\end{enumerate}
The proof is complete.
\end{proof}

\end{document}